\documentclass[11pt,leqno,amscd,amssymb,verbatim, url]{amsart}
\usepackage{amsfonts,amssymb,graphicx}
\usepackage{amsmath,amscd}
\usepackage{hyperref}

\let\oldsection=\section

\newcommand{\gfpr}{G(\mathbb{F}_{q})}

\newcommand{\sG}{{\mathcal G}}

\newcommand{\ind}{\operatorname{ind}}
\newcommand{\Ext}{\operatorname{Ext}}

\newcommand{\opH}{\operatorname{H}}

\newcommand{\Hom}{\operatorname{Hom}}

\newcommand{\res}{{\text{\rm res}}}
\newcommand{\Map}{{\text{\rm Map}}}
\newcommand{\nopH}{\opH}

\newcommand{\la}{\lambda}

\newcommand{\bN}{{\mathbb Z}}

\theoremstyle{definition}

\newtheorem{theorem}{Theorem}[subsection]

\newtheorem{question}{Question}[subsection]
\makeatletter\let\c@question\c@theorem\makeatother

\makeatletter\let\c@fact\c@theorem\makeatother

\makeatletter\let\c@note\c@theorem\makeatother

\newtheorem{lemma}{Lemma}[subsection]
\makeatletter\let\c@lemma\c@theorem\makeatother

\makeatletter\let\c@alg\c@theorem\makeatother

\newtheorem{remark}{Remark}[subsection]
\makeatletter\let\c@remark\c@theorem\makeatother

\newtheorem{example}{Example}[subsection]
\makeatletter\let\c@example\c@theorem\makeatother

\newtheorem{prop}{Proposition}[subsection]
\makeatletter\let\c@prop\c@theorem\makeatother

\makeatletter\let\c@conj\c@theorem\makeatother

\newtheorem{cor}{Corollary}[subsection]
\makeatletter\let\c@cor\c@theorem\makeatother

\makeatletter\let\c@defn\c@theorem\makeatother

\numberwithin{equation}{subsection}

\usepackage[capitalise]{cleveref}
\crefname{theorem}{Theorem}{Theorems}
\crefname{fact}{Fact}{Facts}
\crefname{question}{Question}{Question}
\crefname{note}{Note}{Notes}
\crefname{lemma}{Lemma}{Lemmas}
\crefname{alg}{Algorithm}{Algorithms}
\crefname{remark}{Remark}{Remarks}
\crefname{example}{Example}{Examples}
\crefname{prop}{Proposition}{Propositions}
\crefname{conj}{Conjecture}{Conjectures}
\crefname{cor}{Corollary}{Corollaries}
\crefname{defn}{Definition}{Definitions}
\crefname{equation}{}{}

\begin{document}
\title{Bounding Extensions for Finite Groups and Frobenius Kernels}

\author{\sc C.P. Bendel}
\address
{Department of Mathematics, Statistics and Computer Science\\
University of
Wisconsin-Stout \\
Menomonie\\ WI~54751, USA}
\thanks{}
\email{bendelc@uwstout.edu}

\author{\sc D.K. Nakano}
\address
{Department of Mathematics\\ University of Georgia \\
Athens\\ GA~30602, USA}
\thanks{Research of the second author was supported in part by NSF
grant DMS-1002135}
\email{nakano@math.uga.edu}

\author{\sc B.J. Parshall}
\address
{Department of Mathematics\\ University of Virginia \\
Charlottesville\\ VA~22903, USA}
\thanks{Research of the third author was supported in part by NSF
grant DMS-1001900}
\email{ bjp8w@virginia.edu}

\author{\sc C. Pillen}
\address{Department of Mathematics and Statistics \\ University of South
Alabama\\
Mobile\\ AL~36688, USA}
\email{pillen@southalabama.edu}

\author{\sc L.L. Scott}
\address
{Department of Mathematics\\ University of Virginia \\
Charlottesville\\ VA~22903, USA}
\thanks{Research of the fifth author was supported in part by NSF
grant DMS-1001900}
\email{lls2l@virginia.edu}

\author{\sc D. Stewart}
\address
{New College\\ University of Oxford \\
Oxford\\ UK}
\email{david.stewart@new.ox.ac.uk}


\begin{abstract} Let $G$ be a simple, simply connected algebraic group 
defined over an algebraically closed field $k$ of positive characteristic $p$. Let $\sigma:G\to G$ be
a strict endomorphism (i.~e., the subgroup $G(\sigma)$ of $\sigma$-fixed points is finite).
Also, let $G_\sigma$ be the scheme-theoretic kernel of $\sigma$, an infinitesimal subgroup of $G$. 
This paper shows that the degree $m$ cohomology $\opH^m(G(\sigma),L)$ of any irreducible $kG(\sigma)$-module $L$ is bounded by a constant depending on the root system $\Phi$ of $G$ and the integer $m$. A similar result
holds for the degree $m$ cohomology of $G_\sigma$.   These bounds are actually established for
the degree $m$ extension groups $\Ext^m_{G(\sigma)}(L,L')$ between irreducible $kG(\sigma)$-modules $L,L'$, with again a similar result holding for $G_\sigma$.  In these $\Ext^m$ results, of interest
in their own right, the bounds depend also on $L$, or, more precisely, on length of the $p$-adic expansion of the highest weight
associated to $L$.  All bounds are, nevertheless,
independent of the characteristic $p$. These results extend earlier work of Parshall and Scott for
rational representations of algebraic groups $G$.

We also show that one can find bounds independent of the prime 
for the Cartan invariants of $G(\sigma)$ and $G_{\sigma}$, and even for the lengths
of the underlying PIMs.  These bounds, which depend only on the root system of $G$ and 
the ``height" of $\sigma$, 
provide in a strong way an affirmative answer to a question of Hiss, for the special case of finite groups $G(\sigma)$ of Lie 
type in the defining characteristic. 
\end{abstract}

\maketitle
\section{Introduction}
\subsection{Overview} Let $H$ be a finite group and let $V$ be a faithful, absolutely irreducible $H$-module over a field $k$ of positive characteristic. It has long been empirically observed  that the dimensions of the $1$-cohomology groups $\opH^1(H,V)$ are small. Formulating appropriate statements to capture this intuition and explain this phenomenon has formed a theme in group theory for the past thirty years. Initially, most of the work revolved around finding bounds for $\dim \opH^1(H,V)$ in terms of $\dim V$ and can be ascribed to Guralnick and his collaborators. One result in this direction occurs in \cite{GH98}: if $H$ is quasi-simple, then \begin{equation}\dim \opH^1(H,V)\leq \frac{1}{2}\dim V.\label{gh}\end{equation} However, for large values of $\dim V$, this bound is expected to be vastly excessive---to such a degree that in the same paper, the authors go so far as to conjecture a universal bound on $\dim \opH^1(H,V)$ is 2 (\cite[Conjecture 2]{GH98}). While the existence of 3-dimensional $\opH^1(H,V)$ were found by Scott \cite{Sco03}, the original question of Guralnick in \cite{Gur86} of finding some universal bound on $\dim\opH^1(H,V)$ remains open, as does the analogous question \cite[Question 12.1]{GKKL07} of finding numbers $d_m>0$ bounding $\dim \opH^m(H,V)$ when $H$ is required to be a finite simple group. However, it currently appears that (much) larger values of $\dim\opH^1(H,V)$ will be soon established\footnote{For $H=PSL(8,p)$  for $p$ large, there are values
$\dim \opH^1(H,V)=469$ with $V$ absolutely irreducible. This is based on computer calculations of
F. Luebeck, as confirmed by a student T. Sprowl of Scott, both using the
method of \cite{Sco03}. Luebeck computed
similarly large dimensions for $F_4(p)$ and $E_6(p)$. Until recently, the largest
known dimension (with faithful absolutely irreducible coefficients) had been $3$.} and that even to this more basic question, the answer will almost certainly be ``no" for any $m>0$.

For this reason, bounds for $\dim \opH^m(H,V)$ which depend on the group $H$ rather than the module $V$ are all the more indispensable.

Recent progress towards this goal was made by Cline, Parshall, and Scott.\footnote{In describing this work, we freely use the standard notation in Section 2.1.} If $G$ is a simple algebraic group over an algebraically
closed field $k$ of
positive characteristic, we call a surjective endomorphism $\sigma:G\to G$ {\it strict} provided the
group $G(\sigma)$ of $\sigma$-fixed points is finite.\footnote{These endomorphisms are studied
extensively by Steinberg \cite{Stein68}, though he did not give them a name.}

\begin{theorem}\label{h1} (\cite[Thm.~7.10]{CPS09}) Let $\Phi$ be a finite irreducible root system. There is a constant $C(\Phi)$ depending only on $\Phi$ with the following property. If $G$ is any simple, simply connected algebraic group over an algebraically closed field $k$ of positive characteristic  with
root system $\Phi$, and if $\sigma:G\to G$ is a strict endomorphism, then $\dim \opH^1(G(\sigma), L)\leq	 C(\Phi)$, for all irreducible $kG(\sigma)$-modules $L$.\end{theorem} 

Therefore, the numbers $\dim\opH^1(H,V)$ are universally bounded for all finite groups $H$ of Lie type of 
a fixed Lie rank and all irreducible $kH$-modules $V$ in the defining characteristic.\footnote{In this paper,
we consider the finite groups $G(\sigma)$, where $G$ is a simple, simply connected algebraic group
over an algebraically closed field $k$ of positive characteristic $p$ and $\sigma$ is a strict endomorphism. The passage to semisimple groups is routine and omitted. For a precise definition
of a ``finite group of Lie type," see \cite[Ch. 2]{GLS98}. It is easy to bound the cohomology (in the
defining characteristic) of a simple finite
group $H$ of Lie type
 in terms of the bounds on the dimension of the cohomology of a group $G(\sigma)$ in which $H$
 appears as a section.  More generally, if
$H_2$ is a normal subgroup of a finite group $H_1$, and $H_3:=H_2/N$ for $N\trianglelefteq H_2$ of order
prime to $p$, then, for any fixed $n$ and irreducible $kH_3$-module $L$, $\dim\opH^n(H_3,L)$ is bounded
by a function of the index $[H_1:H_2]$, and the 
maximum of all $\dim\opH^m(H,L')$, where $L'$ ranges over $kH_1$-modules and $m\leq n$. A similar statement holds for $\Ext^n$ for irreducible modules. Taking $H_3=H$ and $H_1=G(\sigma)$, 
questions of cohomology or Ext-bounds for $H$ can be reduced to corresponding bounds for $G(\sigma)$. 
 The argument involves induction from $H_2$ to $H_3$,  and standard
Hochschild-Serre spectral sequence methods. Further details are left to the reader.  In the $\Ext^n$ case, one may be equally interested in the groups $\Ext^n_{\widetilde H}(\widetilde L,\widetilde L')$ where 
$\widetilde H$ is a covering group of the finite simple group $H$ of Lie type and $\widetilde L, \widetilde L'$ are irreducible defining characteristic modules
for $\widetilde H$. In all but finitely many cases, see \cite[Ch. 6]{GLS98}, $\widetilde H$ is a homomorphic
image of $G(\sigma)$ with kernel central and of order prime to $p$, so the discussion above applies
as well in this case to bound $\dim\Ext^n_{\widetilde H}(\widetilde L, \widetilde L')$, using corresponding
bounds for $G(\sigma)$. 
 }   Subsequently, the cross-characteristic case was handled in Guralnick and Tiep in \cite[Thm.~1.1]{GT11}. 
Thus, combining this theorem with Theorem \ref{h1}, there is a constant $C_r$ depending only on the Lie rank $r$ such that for a finite simple group $H$ of Lie type of Lie rank $r$, $\dim\opH^1(H,V)<C_r$, for
all irreducible $H$-modules $V$ over any algebraically closed field (of arbitrary characteristic). It is worth noting that this bound $C_r$ affords an improvement on the bound in display (\ref{gh}) almost all the time, insofar as it is an improvement for all modules whose dimensions are bigger than $C_r/2$.  

In later work \cite[Cor. 5.3]{PS11},  Parshall and Scott proved a stronger result which states that, under the same assumptions, there exists a constant $C'=C'(\Phi)$ bounding the dimension of $\Ext^1_{G(\sigma)}(L,L')$ for all irreducible $kG(\sigma)$-modules (in the defining characteristic).   The proofs of both this $\Ext^1$-result and the above $\opH^1$-result proceed along the similar general lines of finding bounds for the dimension of $\Ext_G^1(L,L')$ and then using  the result \cite[Thm.~5.5]{BNP06} of Bendel, Nakano, and Pillen to relate $G$-cohomology to $G(\sigma)$-cohomology. Specific calculations of Sin \cite{Sin94}
were needed to handle the Ree and Suzuki groups. 

Much more is known in the algebraic group 
case. Recall that the rational irreducible modules for $G$ are parametrized by the set $X^+=X^+(T)$ of dominant weights for $T$ a maximal torus of $G$. For a non-negative integer $e$, let $X_e$ denote the set of $p^e$-restricted dominant weights (thus, $X_0:=\{0\}$).

\begin{theorem}\label{extalg} (\cite[Thm.~7.1, Thm.~5.1]{PS11}) Let $m,e$ be nonnegative integers and $\Phi$ be a finite irreducible root system. There exists a constant $c(\Phi,m,e)$ with the following property. Let $G$  be a simple,
simply connected algebraic group defined over an algebraically closed field $k$ of positive characteristic $p$
with root system $\Phi$. If $\lambda,\nu\in X^+$ with $\lambda\in X_e$, then  
\[
\dim\Ext^m_G(L(\lambda),L(\nu))=\dim\Ext^m_G(L(\nu),L(\lambda))\leq c(\Phi,m,e).
\]
In particular, $\dim \opH^m(G,L(\nu))\leq c(\Phi,m,0)$ for all $\nu\in X^+$.\footnote{In fact, both $m$ and $e$ are known to be necessary when $m>1$ (see \cite{Ste12}).}\end{theorem}

A stronger result holds in the $m=1$ case. 
\begin{theorem}\label{extalgm=1}(\cite[Thm.~5.1]{PS11})  There exists a constant
$c(\Phi)$ with the following property. If $\lambda,\mu\in X^+$ , then
 $$\dim\Ext^1_G(L(\lambda),L(\mu))\leq c(\Phi)$$
 for any simple, simply connected  algebraic group $G$ over an algebraically closed field with root system $\Phi$. \end{theorem}
 
In the same paper (see \cite[Cor. 5.3]{PS11}) the aforementioned result was proved for the finite groups $G(\sigma)$. A main
 goal of the present paper amounts to extending Theorem \ref{extalg} to the finite group $G(\sigma)$ and
 irreducible modules in the defining characteristic.
 There are various reasons one wishes to obtain such analogs, along the lines of Theorem \ref{h1}.  The cases $m=2$ and $\lambda=0$ are especially important.  For example, the  second cohomology group $\opH^2(H,V)$ parametrizes non-equivalent group extensions of $V$ by $H$; it is also intimately connected to the lengths of profinite  presentations, a fact that \cite{GKKL07} presses into service. At this point, it is worth mentioning a conjecture of Holt, made a theorem by Guralnick, Kantor, Kassabov and Lubotsky:

\begin{theorem}[{\cite[Thm.~B$'$]{GKKL08}}]\label{gkkl} There is a constant $C$ so that 
\[\dim \opH^2(H,V)\leq C\dim V\]
 for any quasi-simple group $H$ and any absolutely irreducible $H$-module $V$.\footnote{It is shown in \cite[Thm.~B]{GKKL07} that one can take $C=17.5$.
}\end{theorem}

Suppose one knew, as in Theorem \ref{h1}, that there were a constant $c'=c'(\Phi)$ so that $\dim \opH^2(G(\sigma),L)\leq c'$. Then for a group $G(\sigma)$ of fixed Lie rank in defining characteristic\footnote{As the order of the Sylow $p$-subgroups of $G(\sigma)$ are by far the biggest when $p$ is the defining characteristic of $G(\sigma)$, one very much expects this case to give the largest dimensions of $\opH^2(G(\sigma),V)$, hence it will be the hardest to bound.}, one would have as before that this bound would be better than that proposed by Theorem \ref{gkkl} almost all the time.

One purpose of this paper is to demonstrate the existence of such a constant; moreover, we achieve an exact analog to Theorem \ref{extalg} for finite groups of Lie type. The methods are sufficiently powerful to obtain a number of other interesting results.


\subsection{Bounding ${\text{\rm \bf {Ext}}}$ for finite groups of Lie type}
Theorem \ref{thm:finitebound} below is a central result of this  paper. It gives bounds for the higher extension
groups of the  finite groups  $G(\sigma)$, where $\sigma$ is a strict endomorphism of a  simple,  simply connected algebraic group $G$ over a field of positive
characteristic, and the coefficients are irreducible modules in the defining characteristic. In
this case,  the irreducible modules $G(\sigma)$-modules are parametrized by the set $X_\sigma$ of $\sigma$-restricted dominant weights. 

 \begin{theorem}\label{thm:finitebound} Let $e,m$ be non-negative integers and let $\Phi$ be a finite  irreducible root system. Then there exists a constant $D(\Phi,m,e)$, depending only on $\Phi$, $m$ and $e$ 
(and not on any field characteristic $p$) with the following property. Given any simple, simply connected
algebraic group $G$ over an algebraically closed field $k$ of positive characteristic $p$ with root system $\Phi$, and given any strict endomorphism $\sigma$ of $G$ such that
$X_e\subseteq X_\sigma$,\footnote{The condition that $X_e\subseteq X_\sigma$ merely guarantees that $L(\lambda)$,
$\lambda\in X_e$, restricts to an irreducible $G(\sigma)$-module. In most cases, $\sigma$ is simply a Frobenius map (either standard or twisted with a graph automorphism). If $p=2$ and $G=C_2$ or $F_4$ or if $p=3$ and $G=G_2$, there are more options for $\sigma$ corresponding to Ree and Suzuki groups. See Section~\ref{fgintro} for more details.}
 then for $\lambda\in X_e,\mu\in X_\sigma$, we have
$$\dim\Ext^m_{G(\sigma)}(L(\lambda),L(\mu))\leq D(\Phi,m,e).$$ In particular, 
$$\dim\opH^m(G(\sigma),L(\lambda))\leq D(\Phi,m,0)$$ 
for all $\lambda\in X_{\sigma}$.
\end{theorem}

 Let us outline the proof. Following the ideas first introduced by Bendel, Nakano, and Pillen,  $\Ext$-groups for the finite groups of Lie type $G(\sigma)$ to $\Ext$-groups for an ambient algebraic group $G$.  One has, by generalized Frobenius reciprocity, that 
 \begin{equation}\label{gfr}\Ext^m_{G(\sigma)}(L(\lambda),L(\mu))\cong \Ext^m_{G}(L(\lambda),L(\mu)\otimes \ind_{G(\sigma)}^Gk).\end{equation}
  One key step in the proof is to show (in \S\ref{filtsec}) that the induced $G$-module $\ind_{G(\sigma)}^Gk$ has a filtration with sections of the form $\nopH^{0}(\lambda)\otimes\nopH^{0}(\lambda^\star)^{(\sigma)}$, with each $\lambda\in X^+$ appearing exactly once. If $G_\sigma$ denotes the scheme-theoretic kernel of the map $\sigma:G\to G$, we investigate the right hand side of (\ref{gfr}) using the Hochschild--Serre spectral sequence corresponding to $G_\sigma\triangleleft G$. Bounds on the possible weights occurring in $\Ext^m_{G_\sigma}(L(\lambda),L(\mu))$ (Theorem \ref{thm:Gr-weightbound}) allow us to see (in Theorem \ref{existsAnF}) that cofinitely many of the sections occurring in $\ind_{G(\sigma)}^Gk$ contribute nothing to the right hand side of (\ref{gfr}), so $\ind_{G(\sigma)}^Gk$ can be replaced in that expression with a certain finite dimensional rational $G$-module. This, together with a result bounding the composition factor length of tensor products (Lemma \ref{compFacts}), provides the ingredients to prove Theorem \ref{thm:finitebound}. We show, in fact, that the maximum of $\dim\Ext_{G_\sigma}^m(L(\lambda),L(\mu))$ is bounded above by a certain multiple of $c(\Phi,m, e')$, where $c(\Phi,m, e')$ is the integer coming from Theorem \ref{extalg}.


\subsection{Bounding ${\text{\rm \bf {Ext}}}$ for Frobenius Kernels}\label{SS:boundforFrobenius}

Let $G$ be a simple, simply connected group over an algebraically closed field of positive
characteristic. For a strict endomorphism $\sigma:G\to G$, let $G_\sigma$ denote (as before) its scheme-theoretic kernel.  
The main result in \S\ref{S:FrobeniusKernels} is the proof of the following.

\begin{theorem}\label{thm:boundforFrobeniusKernel}  Let $e,m$ be non-negative integers and let $\Phi$ be a finite  irreducible root system. Then there exists a constant $E(\Phi,m,e)$ (resp., $E(\Phi)$), depending only on $\Phi$, $m$ and $e$ (resp.,  $\Phi$) 
(and not on any field characteristic $p$) with the following property. Given any simple, simply connected
algebraic group over an algebraically closed field $k$ of positive characteristic $p$ with root system $\Phi$, and given any strict endomorphism $\sigma$ of $G$ such that
$X_e\subseteq X_\sigma$, then for $\lambda\in X_e,\mu\in X_\sigma$, 
 $$
\dim \operatorname{Ext}^{m}_{G_\sigma}(L(\lambda),L(\mu))\leq E(\Phi,m,e).
$$ 
In particular, 
$$\dim \operatorname{H}^{m}(G_\sigma,L(\la))\leq E(\Phi,m,0)$$
 for all
 $\la \in X_\sigma$. Furthermore, 
$$\dim\Ext^1_{G_\sigma}(L(\lambda),L(\mu)) \leq E(\Phi)$$
for all $\la, \mu\in X_\sigma$.
\end{theorem}

The proof proceeds by investigating the induced module $\ind_{G_\sigma}^Gk$. This time there is a filtration of $\ind_{G_\sigma}^Gk$ by sections of the form $(\nopH^0(\nu)^{(\sigma)})^{\oplus \dim\nopH^0(\nu)}$. Again, only cofinitely many of these sections contribute to 
\[\Ext^i_{G_\sigma}(L(\lambda),L(\mu))\cong \Ext^i_G(L(\lambda),L(\mu)\otimes \ind_{G_\sigma}^G(k)),\] 
so $\ind_{G_\sigma}^Gk$ can be replaced on the right hand side by a finite dimensional rational $G$-module.

Sections \S\ref{examples:boundforFrobenius} and \S\ref{lowerBound} provide various examples to show that Theorem \ref{thm:boundforFrobeniusKernel} cannot be improved upon. In particular, Theorem \ref{thm:lowerboundforH1} shows that the inequality 
$$\operatorname{max}\{\dim \operatorname{H}^{1}(G_{r},L(\lambda)):\ \lambda\in X_r\}\geq \dim V$$ 
holds,
where $V$ is an irreducible non-trivial finite dimensional rational $G$-module of smallest dimension. Thus, the generalization of the Guralnick conjecture in \cite{Gur86} to general finite group schemes cannot hold.


\subsection{Cartan Invariants}
Let again $\sigma$ denote a strict endomorphism of a simple, simply connected algebraic group $G$, so that $G(\sigma)$ is a finite group of Lie type. In addition to  cohomology, we tackle the related question of bounding the Cartan invariants $[U_\sigma(\lambda):L(\mu)]$, where $U_\sigma(\lambda)$ denotes the projective cover of a irreducible module $L(\lambda)$ for a finite group of Lie type $G(\sigma)$.

If $H$ is a finite group and $k$ is an algebraically closed field, let $c(kH)=\max\{\dim \Hom_{kH}(P,Q)\}$ where the maximum is over all $P$ and $Q$ which are principal indecomposable modules for the group algebra $kH$. Then $c(kH)$ is the maximum Cartan invariant for $kH$. The following question has been raised by Hiss \cite[Question 1.2]{His00}.

 \begin{question}\label{hissq} Is there a function $f_p:\bN\to\bN$ such that $c(kH)\leq f_p(\log_p|H|_p)$
 for all finite groups $H$ and all algebraically closed fields $k$ of characteristic $p$?\end{question}

We provide in our context of finite groups of Lie type in their defining characteristic representations an
answer to the above question with $f_p=f$, independent of $p$.\footnote{In the language of block theory
for finite groups, we have bounded these Cartan invariants not only by a function of the defect group, but by
a function of the defect itself.}   This may be viewed as complementary
to many of the discussions in \cite{His00} for non-defining characteristic. As pointed out by Hiss, the issue of bounding Cartan invariants is related to many other questions in modular representation theory of finite groups, such as the Donovan conjecture.\footnote{This states that, given a finite $p$-group $P$, there are, up to Morita equivalence, only finitely many blocks of group algebras in characteristic $p$  having defect group $P$. As Hiss points out, the Morita equivalence class is determined by 
a basic algebra, and there are only finitely many possibilities for the latter when its (finite) field of definition is known, and
the Cartan matrix entries are bounded. The size of the Cartan matrix is bounded in terms of the defect group order, for any block of a finite group algebra, by a theorem of Brauer and Feit. 
In more direct applications, bounds on Cartan invariants for finite group algebra blocks give bounds on decomposition numbers, through the equation $C=D^t\cdot D$, relating the decomposition matrix $D$ to the Cartan matrix $C$. The equation also shows that the number of ordinary characters can be bounded, using a bound for the size of Cartan matrix
entries, since a bound on the number of Brauer characters (size of the Cartan matrix) is available. It is a still open conjecture of Brauer that the number of ordinary characters is also bounded by (the order of) the
defect group.} For later results in the area of representations in non-defining characteristic, see Bonaf\'e--Rouqui\"er \cite{BR03}.
 
As described more completely in Section 2.3 below, any strict endomorphism $\sigma:G\to G$ involves a  ``power" $F^s$ of
the Frobenius morphism on $G$, where $s$ is a positive integer, except, in the cases of the Ree and Suzuki groups, it is allowed to be half an odd integer. We call $s$ the {\it height} of $\sigma$. 

\begin{theorem}\label{thm:Cartan-finite} Let $s$ be non-negative integer or half an odd positive integer,   and let $\Phi$ be an irreducible root system. Then there exists a constant $N(\Phi,s)$ such that for any simple, simply connected algebraic group $G$ with root system $\Phi$ 
\[
[U_\sigma(\la):L(\mu)|_{G(\sigma)}] \leq N(\Phi,s)
\]
for all $\la, \mu \in X_\sigma$ and any strict endomorphism $\sigma$ of height $s$.
\end{theorem}
In the case that $H=G(\sigma)$, we have that $\log_p(|H|_p)$ is independent of $p$. In fact, it is bounded above by $\log_p((p^s)^{|\Phi|^+})=s|\Phi^+|$ where $s$ is the height of $\sigma$. Now, our theorem provides a function answering Hiss's question: Take $f_p(m)=\max_{r|\Phi^+|\leq m} N(\Phi,s)$ as in Theorem \ref{thm:Cartan-finite}, where again $s$ is the height of $\sigma$. Then we have $c(kG)\leq f_p(m)$ as required. In fact, we get an especially strong answer to Question \ref{hissq} since our function $f_p$ is actually independent of $p$, so can be replaced by a universal function $f$.

The process of proving Theorem \ref{thm:Cartan-finite} leads to an even stronger result, bounding
the composition factor length of the PIMs for $G(\sigma)$ and for $G_\sigma$. This result is stated formally in 
Corollary \ref{lastcor}. The analog of Theorem \ref{thm:Cartan-finite} for $G_\sigma$ 
is proved in the same section,
as  is Theorem \ref{thm:Cartan-Frobenius}.


\section{Preliminaries}
\subsection{Notation}  Throughout this paper, the following basic notation will be used. In many cases,
the decoration ``$r$" (or ``$q$") used for split Chevalley groups has an analog ``$\sigma$" for twisted Chevalley
groups, as indicated.

\begin{enumerate}

\item[(1)] $k$: an algebraically closed field of characteristic $p>0$.

\item[(2)] $G$: a simple, simply connected algebraic group which is defined 
and split over the finite prime field ${\mathbb F}_p$ of characteristic $p$. The assumption
that $G$ is simple (equivalently, its root system $\Phi$ is irreducible) is largely one of convenience.
All the results of this paper extend easily to the semisimple, simply connected case.

\item[(3)] $F:G\rightarrow G$: the Frobenius morphism. 

\item[(4)] $G_r=\text{ker }F^{r}$: the $r$th Frobenius kernel of $G$. More generally, if $\sigma:G\to G$
is a surjective endomorphism, then $G_\sigma$ denotes the scheme-theoretic kernel of $\sigma$.

\item[(5)] $G({\mathbb F}_{q})$: the associated finite Chevalley group. More generally, if $\sigma:G\to G$
is a surjective endomorphism, $G(\sigma)$ denotes the subgroup of $\sigma$-fixed points.

\item[(6)] $T$: a maximal split torus in $G$. 

\item[(7)] $\Phi$: the corresponding (irreducible) root system associated to $(G,T)$.

\item[(8)] $\Pi=\{\alpha_1,\cdots,\alpha_n\}$: the set of simple roots (Bourbaki ordering).  

\item[(9)] $\Phi^{\pm}$: the positive (respectively, negative) roots.  

\item[(10)] $\alpha_0$: the maximal short root.

\item[(11)]  $B$: a Borel subgroup containing $T$ corresponding to the negative roots. 

\item[(12)] $U$: the unipotent radical of $B$.

\item[(13)] $\mathbb E$: the Euclidean space spanned by $\Phi$ with inner product $\langle\,,\,\rangle$ normalized so that $\langle\alpha,\alpha\rangle=2$ for $\alpha \in \Phi$ any short root.

\item[(14)] $\alpha^\vee=2\alpha/\langle\alpha,\alpha\rangle$: the coroot of $\alpha\in \Phi$.

\item[(15)] $\rho$: the Weyl weight defined by $\rho=\frac{1}{2}\sum_{\alpha\in\Phi^+}\alpha$.

\item[(16)] $h$: the Coxeter number of $\Phi$, given by $h=\langle\rho,\alpha_0^{\vee} \rangle+1$.

\item[(17)] $W=\langle s_{\alpha_1},\cdots,s_{\alpha_n}\rangle\subset{\mathbb O}({\mathbb E})$: the Weyl group of $\Phi$, generated by the orthogonal reflections $s_{\alpha_i}$, $1\leq i\leq n$. For $\alpha\in \Phi$, $s_\alpha:{\mathbb E} \to{\mathbb E}$ is the orthogonal reflection in the hyperplane $H_\alpha\subset \mathbb E$ of vectors orthogonal to $\alpha$.

\item[(18)] $W_p=p Q \rtimes W$: the affine Weyl group, where $Q={\mathbb Z}\Phi$ is the root lattice, is generated by the affine reflections $s_{\alpha,pr}:{\mathbb E}\to{\mathbb E}$ defined by $s_{\alpha,rp}(x)=x-[\langle x,\alpha^\vee\rangle-rp]\alpha$, $\alpha\in\Phi$, $r\in\mathbb Z$. Here $p$ can be a positive integer.
 $W_p$ is a Coxeter group with fundamental system $S_p=\{s_{\alpha_1},\cdots, s_{\alpha_n}\}\cup\{s_{\alpha_0,-p}\}$.
 
\item[(19)] $l:W_p\to\mathbb N$: the usual length function on $W_p$.

\item[(20)] $X=\mathbb Z \varpi_1\oplus\cdots\oplus{\mathbb Z}\varpi_n$: the weight lattice, where the fundamental dominant weights $\varpi_i\in{\mathbb E}$ are defined by $\langle\varpi_i,\alpha_j^\vee\rangle=\delta_{ij}$, $1\leq i,j\leq n$.

\item[(21)] $X^+={\mathbb N}\varpi_1+\cdots+{\mathbb N}\varpi_n$: the cone of dominant weights.

\item[(22)] $X_{r}=\{\lambda\in X^+: 0\leq \langle\lambda,\alpha^\vee\rangle<p^{r},\,\,\forall \alpha\in\Pi\}$: the set of $p^{r}$-restricted dominant weights. As discussed later, if $\sigma:G\to G$ is a surjective endomorphism,
$X_\sigma$ denotes the set of $\sigma$-restricted dominant weights.

\item[(23)] $\leq$, $\leq_{Q}$ on $X$: a partial ordering of weights, for $\lambda, \mu \in X$, $\mu\leq \lambda$ (respectively $\mu\leq_{Q} \lambda$) if and only if $\lambda-\mu$ is a linear combination 
of simple roots with non-negative integral (respectively, rational) coefficients. 


\item[(24)] $\lambda^\star : = -w_0\lambda$: where $w_0$ is the longest word in the Weyl group $W$ and $\lambda\in X$. 

\item[(25)] $M^{(r)}$:  the module obtained by composing the underlying representation for 
a rational $G$-module $M$ with $F^{r}$.
More generally, if $\sigma:G\to G$ is a surjective endomorphism $M^{(\sigma)}$ denotes the module
obtained by composing the underlying representation for $M$ with $\sigma$.

\item[(26)] $\nopH^0(\lambda) := \operatorname{ind}_B^G\lambda$, $\lambda\in X^+$: the induced module whose character is provided by Weyl's character formula.  

\item[(27)] $V(\lambda)$, $\lambda\in X^+$: the Weyl module of highest weight $\lambda$. Thus, $V(\lambda)\cong \nopH^0(\lambda^\star)^*$.

\item[(28)] $L(\lambda)$: the irreducible finite dimensional $G$-module with highest weight $\lambda$. 

\end{enumerate}


\subsection{Finite groups of Lie type}\label{fgintro} This subsection sets the stage for studying the cohomology
of the finite groups of Lie type. That is,   we consider the groups $G(\sigma)$ of 
 $\sigma$-fixed points for a strict endomorphism $\sigma:G\to G$. By definition, $G(\sigma)$ is
 a finite group.   There are three cases to consider.
 
 \begin{itemize}
 \item[(I)] The finite Chevalley groups $G({\mathbb F}_q)$. For a positive integer $r$, let $q=p^r$ and set
  $G(F^r)=G({\mathbb F}_q)$, the group of $F^r$-fixed points. Thus, $\sigma=F^r$ in this case.
  
  \item[(II)] The twisted Steinberg groups. Let $\theta$ be a non-trivial graph automorphism of $G$ stabilizing $B$ and $T$. 
 For a positive integer $r$, set $\sigma=F^r\circ \theta=\theta\circ F^r: G\to G$. Then let $G(\sigma)$ be the
  finite group of $\sigma$-fixed points.  Thus, $G(\sigma)={^2A}_n(q^2)$, $^2D_n(q^2)$, $^3D_4(q^3)$, or 
  $^2E_6(q^2)$.

  \item[(III)] The Suzuki groups and Ree groups.  Assume that $G$ has type $C_2$ or $F_4$ and $p=2$
  or that $G$ has type $G_2$ and $p=3$. Let  $F^{1/2}:G\to G$ be a fixed purely inseparable isogeny 
  satisfying $(F^{1/2})^2=F$; we do not repeat the explicit description of $F^{1/2}$, but instead refer to
  the lucid discussion given in \cite[I, 2.1]{SpSt70}.
  For an odd positive integer $r$, set $\sigma=F^{r/2}=(F^{1/2})^r$. 
  Thus, $G(\sigma) ={^2C}_2(2^{\frac{2m+1}{2}})$,
  ${^2F}_4(2^{\frac{2m+1}{2}})$, or, $^2G_2(3^{\frac{2m+1}{2}})$. Both here and in (II), we follow the notation suggested in 
  \cite{GLS98}.\footnote{For a discussion of the differences between
  the simply connected and adjoint cases in case (III), see \cite[p. 1012]{Sin94}.}
   
   \end{itemize}
   
 In all the above cases, the group scheme-theoretic kernel $G_\sigma$ of $\sigma$ plays an important role. In
 case (I), where $\sigma=F^r$, this kernel is commonly denoted $G_r$, and it is called the $r$th Frobenius kernel. 
 In case (II),  with $\sigma=F^r\circ\theta$, $\theta$ is an automorphism
so that $G_\sigma=G_r$. In case (III), with $\sigma=F^{r/2}=
 (F^{1/2})^r$, with $r$ an odd positive integer, we often denote
 $G_\sigma$ by $G_{r/2}$ . For example, $G_{1/2}$ has coordinate algebra $k[G_{1/2}]$,  
 the dual of the restricted enveloping algebra of the subalgebra (generated by the short root spaces) of the Lie algebra of $G$ generated
 by the short simple roots. 
 
\begin{remark}\label{generalFrobenius} (a) The Frobenius kernels $G_r$ play a central role in
 the representation theory of $G$; see, for example, Jantzen \cite{Jantzen} for an exhaustive
 treatment. These results are all available in cases (I) or (II). But many standard results using $G_r$ hold equally well for the more exotic
 infinitesimal subgroups $G_{r/2}$ in case (III), which we now discuss. Suppose $r=2m+1$ is an odd positive 
 integer and $\sigma=F^{r/2}$. For a rational $G$-module $M$, let $M^{(r/2)}=M^{(\sigma)}$ be the
 rational module obtained by making $G$ act on $M$ through $\sigma$. Additionally, if $M$ has the form
 $N^{(r/2)}$ for some rational $G$-module $N$, put $M^{(-r/2)}=N$. The subgroup $G_{r/2}$ is a normal
 subgroup scheme of $G$, and, given a rational $G$-module $M$, there is a (first quadrant) Hochschild-Serre spectral sequence
 \begin{equation}\label{HS} E^{i,j}_2=\opH^i(G/G_{r/2},\opH^j(G_{r/2},M))
 \cong \opH^i(G,\opH^j(G_{r/2},M)^{(-r/2)})\Rightarrow \opH^{i+j=n}(G,M)\end{equation}
 computing the rational $G$-cohomology of $M$ in degree $n$ in terms of rational cohomology
 of $G$ and $G_{r/2}$.\footnote{We use here that $G/G_{r/2}\cong G^{(r/2)}$,  where $G^{(r/2)}$ has coordinate
 algebra $k[G]^{(r/2)}$. When $G^{(r/2)}$ is identified with $G$, the rational $G/G_{r/2}$-module $\opH^j(G_{r/2},M)$
 identifies with $\opH^j(G_{r/2},M)^{(-r/2)}$.} Continuing with case (III), given a rational $G_{r/2}$-module $M$, there is a Hochschild-Serre
 spectral sequence (of rational $T$-modules)
 \begin{equation}\label{HS2}
 \begin{aligned} E_2^{i,j}=\opH^i(G_{r/2}/G_{1/2}, \opH^j(G_{1/2},M))&\cong& \opH^i(G_{(r-1)/2},\opH^j(G_{1/2},M)^{(-1/2)})^{(1/2)}\\ 
 &\ & \Rightarrow 
 \opH^{i+j=n}(G_{r/2},M).\end{aligned}\end{equation}
 Since $r$ is odd, $G_{(r-1)/2}$ is a classical Frobenius kernel. One could also replace $G_{1/2}$ above by
 $G_1$, say, in type (III), but usually $G_{1/2}$ is more useful.
 
 In addition, still in type (III), given an odd positive integer  $r$ and $\lambda\in X^+$ (viewed as a one-dimensional rational $B$-module), there is a spectral sequence
 \begin{equation}\label{BorelHS}
 E_2^{i,j}=R^i\ind_{B}^G\opH^j(B_{r/2},\lambda)^{(-r/2)}\Rightarrow \opH^{i+j=n}(G_{r/2}, H^0(\lambda))^{(-r/2)}.
 \end{equation}
 This is written down in the classical $r\in\mathbb N$ case in \cite[II.12.1]{Jantzen}, but the proof
 is a special case of \cite[I.6.12]{Jantzen} which applies in all our cases.  
  
 (b) The irreducible $G(\sigma)$-modules (in the defining characteristic) are the restrictions to $G(\sigma)$
 of the irreducible $G$-modules $L(\lambda)$, where $\lambda$ is a $\sigma$-restricted dominant
 weight.  In cases (I) and (II), these $\sigma$-restricted weights
 are just the $\lambda\in X^+$ such that $\langle\lambda,\alpha^\vee\rangle<p^r$, for all $\alpha\in\Pi$. 
  In addition, any $\lambda\in X^+$ can be uniquely written as $\lambda=\lambda_0+p^r\lambda_1$, where
 $\lambda_0\in X_r$ and $\lambda_1\in X^+$. In case (I), the Steinberg tensor product theorem
 states that $L(\lambda)\cong L(\lambda_0)\otimes L(\lambda_1)^{(r)}$. 
 
 In case (II), let $\sigma^*:X\to X$ be the restriction of the comorphism of $\sigma$ to $X$. Then
 write $\lambda=\lambda_0+\sigma^* \lambda_1$, where $\lambda_0\in X_r$ and $\lambda_1\in X^+$.
 Observe that $\sigma^* =p^r\theta$, where here $\theta$ denotes the automorphism of $X$ induced
 by the graph automorphism. We have $L(\lambda)\cong L(\lambda_0)\otimes L(\theta\lambda_1)^{(r)}$,
 which is Steinberg's tensor product theorem in this case. In this case, we also call the weights in $X_r$
 $\sigma$-restricted (even though they are also $r$-restricted).
 
 In case (III), there is a similar notion of $\sigma$-restricted dominant weights. Suppose
 $r=(2m+1)/2$, then the condition that $\lambda\in X^+$ be $\sigma$-restricted is that $\langle\lambda,\alpha^\vee)<p^{m+1}$ for $\alpha\in\Pi$ short, and
 $<p^{m}$ in case $\alpha\in\Pi$ is long. Any dominant weight $\lambda$ can be uniquely written as
 $\lambda=\lambda_0 + \sigma^* \lambda_1$, where $\lambda_0$ is $\sigma$-restricted and $\lambda_1\in X^+$.
Here $\sigma^*:X\to X$ is the restriction to $X\subset k[T]$ of the comorphism $\sigma^*$ of $\sigma$. Then $L(\lambda)\cong L(\lambda_0)\otimes L(\lambda_1)^{(r/2)}$, which is  Steinberg's tensor product theorem for
case (III).  

(c) In all cases  (I), (II) (III), the set $X_\sigma$ of $\sigma$-restricted dominant weights also indexes the irreducible modules for the infinitesimal
subgroups $G_\sigma$; they are just the restrictions to $G_\sigma$ of the corresponding irreducible $G$-modules.
 \end{remark}


\subsection{Bounding weights}  Following \cite{BNP04-Frob}, set $\pi_{s}=\{\nu\in X^+:\ \langle \nu,\alpha_{0}^{\vee} \rangle <  s\}$
and let ${\mathcal C}_{s}$ be the full subcategory of all finite dimensional $G$-modules whose composition factors $L(\nu)$ have highest weights
lying in $\pi_{s}$.\footnote{ ${\mathcal C}_{s}$ is a highest weight category and equivalent to the module category for a finite dimensional
quasi-hereditary algebra. For two modules in ${\mathcal C}_s$, their $\Ext$-groups can be computed
either in ${\mathcal C}_s$ or in the full category of rational $G$-modules.} The condition that $\nu\in \pi_s$
is just that $\nu$ is $(s-1)$-small in the terminology of \cite{PSS12}.

Now let $\sigma:G\to G$ be one of the strict endomorphisms described in cases (I), (II) and (III)
above. In the result below, we provide information about $G$-composition factors of $\text{Ext}^{m}_{G_{\sigma}}(L(\lambda),L(\mu))^{(-\sigma)}$ where 
$\lambda,\mu\in X_\sigma$. 

\begin{theorem} \label{thm:Gr-weightbound} If $\lambda,\mu\in X_\sigma$ , then
$\operatorname{Ext}^{m}_{G_{\sigma}}(L(\lambda),L(\mu))^{(-\sigma)}$ is
a rational $G$-module in ${\mathcal C}_{s(m)}$ where
$$s(m) =\begin{cases}
1& \text{ if } m=0\\
h & \text{ if } m=1,\,{\text{except possibly when $G=F_4, p=2,\sigma=F^{r/2}$, $r$ odd} }\\ 
3m + 2h -2 &\text{ if } m\geq 0, \text{in all cases (I), (II), (III).}\\
\end{cases}
$$
\end{theorem}

\begin{proof} The case $m=0$ is obvious. If $m=1$, then \cite[Prop. 5.2]{BNP04-Frob}
gives $s(m)=h$ in cases (I) and (II). In case (III), except in $F_4$, we can apply the case (I) result, together
with the explicit calculations in \cite{Sin94}, to again deduce that $s(m)=h$ works. Sin shows in \cite[Lem. 2.1,
Lemma 2.3]{Sin94} that the only nonzero $\Ext^1_{G_{1/2}}(L(\lambda),L(\mu))^{(-1/2)}$, $\lambda,\mu\in X_{1/2}$, are $G$-modules of the form $L(\tau)$, $\tau\in X_{1/2}$. Now apply the spectral sequence (\ref{HS2}). 
In the $m\geq 0$ assertion, we use the proof for \cite[Cor. 3.6]{PSS12} which, because of the discussion
given in Remark \ref{generalFrobenius}, works in all cases.
\end{proof}

\begin{remark}(a) As noted in \cite[Rem. 3.7(c)]{PSS12}, the last bound in the theorem can be improved
in various ways. If $\Phi$ is not of type $G_2$, ``$3m$" can be replaced by ``$2m$". If $p>2$, $m$ may
be replaced by $[m/2]$. Finally,  if $m>1$, another bound in cases (I) and (II) is given in \cite[Prop. 5.2]{BNP04-Frob} as
$s(m)=(m-1)(2h-3)+3$, which is better for small values of $m$ and $h$. 

(b) Suppose that $G$ has type $F_4$, $p=2$ and $\sigma=F^{r/2}$ for some odd integer $r$. 
Here  (and in all case (III) instances) Sin \cite{Sin94} explicitly calculates all $\Ext^1_{G_{1/2}}(L(\lambda),L(\mu))^{(-1/2)}$ as $G$-modules
for $\lambda,\mu\in X_{1/2}$. In all nonzero cases, but one, it is of the form $L(\tau)$ with
$\tau\in X_{1/2}$. The one exception, in the notation of \cite{Sin94} is $\lambda=0$ and $\mu=\varpi_3$,
(i.e., the fundamental dominant weight corresponding to the interior short fundamental root), in which case $\Ext^1_{G_{1/2}}(L(\lambda),L(\mu))^{(-1/2)}\cong k\oplus L(2\varpi_4)$.\footnote{In particular, this $\Ext^1$-module, when untwisted,
does not have a good filtration.} 
\end{remark}


\section{Filtrations of certain induced modules} \label{filtsec}
 
\subsection{Preliminaries} In this section we present a generalization of the filtration theory for induction from $G(\sigma)$ to $G$.  In the classical split case for 
Chevalley groups the theory was first developed by Bendel, Nakano and Pillen (cf. \cite[Prop. 2.2 \& proof]{BNP11}). By \cite[\S10.5]{Stein68}, $G(\sigma)$ is finite if and only if
 the differential $dL$ is surjective at the identity $e\in G$. Here $L:G\to G, x\mapsto \sigma(x)^{-1}x$
 is the Lang map. In addition, the Lang-Steinberg Theorem \cite[Thm.~10.1]{Stein68} states (using our
 notation) that
 if $G(\sigma)$ is finite, then $L$ is surjective.  Recall that an endomorphism $\sigma$ is  strict if and only if  $G(\sigma)$ is finite.
  
If $K,H$ are closed subgroups of an arbitrary affine algebraic group $G$, there is in general no known
Mackey decomposition theorem describing the functor $\res_H^G\ind_K^G $. However, in the very special case in which $|K\backslash G/H|=1$, a Mackey decomposition theorem does 
hold. Namely, 
\begin{equation}\label{Mackey} KH=G\implies \res^G_H\ind_K^G=\ind_{K\cap H}^H,\end{equation}
where $K\cap H:=K\times_GH$ is the scheme-theoretic intersection. 
We refer the reader to \cite[Thm.~4.1]{CPS83} and the discussion there, where it is pointed out that the condition
$KH=G$ need only be checked at the level of $k$-points.
 
Let us now return to the case in which $G$ is a simple, simply connected group. First, there is a
natural action of $G\times G$ on the coordinate algebra $k[G]$ given by $((x,y)\star' f)(g) =
(x\cdot f\cdot y^{-1})(g):= f(y^{-1}gx)$, for $(x,y)\in G\times G, g\in G$.  Then we have the following
lemma.

 \begin{lemma}\label{filtration} (\cite{Kop84}) The rational $(G\times G)$-module $k[G]$ has an increasing filtration $0\subset
 {\mathcal F}'_0\subset{\mathcal F}'_1\subset\cdots$ in which, for $i\geq 0$, ${\mathcal F}'_i/{\mathcal F}'_{i-1}\cong \nopH^{0}(\gamma_i)\otimes\nopH^{0}(\gamma_i^\star)$, $\gamma_i\in X^+$, and
 $\cup{\mathcal F}'_i=k[G]$. Each dominant weight $\gamma \in X^+$ appears precisely once in the list $\{\gamma_0,\gamma_1,\cdots\}$.
  \end{lemma}
  
  Let $k[G]^{(1\times\sigma)}=k[G]_\sigma$ denote the coordinate algebra of $G$ viewed
as a rational $G$-module with $x\in G$ acting as 
$$(x\star f)(g):=(x\cdot f\cdot \sigma(x)^{-1})(g)= f(\sigma(x)^{-1}gx),\quad f\in k[G], g\in G.$$
This is compatible with the action of $G\times G$ on $k[G]$ given by $((x,\sigma(y))\star f)(g)= (x\cdot f\cdot \sigma(y)^{-1})(g)=
f(\sigma(y)^{-1}gx)$.

The following proposition gives a description of an increasing $G$-filtration on $k[G]_{\sigma}$. 

 \begin{prop} Assume $\sigma$ is a strict endomorphism of $G$. 
 Then $\ind_{G(\sigma)}^Gk\cong k[G]_\sigma$. In particular, $\ind_{G(\sigma)}^Gk$ has a $G$-
 filtration with sections of the form $\nopH^{0}(\lambda)\otimes\nopH^{0}(\lambda^\star)^{(\sigma)}$, $\lambda\in X^+$
 appearing exactly once. 
 \end{prop}
 
 \begin{proof} Consider two isomorphic copies of $G$ embedded as closed subgroups of $G\times G$, 
 $$\begin{cases} \Delta:= \{(g,g)\,|\,g\in G\};\\
                            \Sigma:= \{ (g,\sigma(g))\,|\,g\in G\}.
                            \end{cases}
$$
The reader may check that the induced module $\ind_\Delta^{G\times G}k=\Map_\Delta(G\times G,k)$ identifies with the $G\times G$-module $k[G]$ in Lemma \ref{filtration}, through inclusion $G \cong G\times 1 \subseteq G\times G$ into the first factor. That is, the comorphism $k[G \times G] \to k[G]$ induces a $G\times G$-equivariant map when restricted to the submodule $\Map_\Delta(G\times G,k)$ of $k[G\times G]$.   
We now wish to apply (\ref{Mackey}) to $\ind_\Delta^{G\times G}k$ by composing the induction functor with restriction to $\Sigma$. 

Given $(a,b)\in G\times G$, there exists an $x\in G$ such that
$\sigma(x)x^{-1}=ba^{-1}$. Let $y:=x^{-1}a$. Then $(x,\sigma(x))(y,y)=(a,b)$,
 so $\Sigma\Delta=G\times G$. 
 Next, we show that $\Delta\cap\Sigma\cong G(\sigma)$ as group schemes under the isomorphism
 $\Sigma\to G$ which is projection onto the first factor. This is clear at the level of $k$-points, so it
 enough to show the $\Delta\cap\Sigma$ is reduced.  However, one can check that $\Delta\cap\Sigma$ is isomorphic to the scheme $X$ defined by 
  the pull-back diagram (in which ${\mathbf e}$ denotes the trivial $k$-group scheme)
  \[\begin{CD} X @>>> {\mathbf e}\\
  @VVV @VVV\\
  G @>L>> G,\end{CD}\]
 so we must show the closed subgroup scheme $X$ of $G$ is reduced (and hence isomorphic to $G(\sigma)$). However, the Lie algebra of $X$ is a subalgebra of $\text{Lie }G$ and then maps injectively to a subspace
 of $\text{Lie }G$ under $dL$. The commutativity of the above diagram implies that $X$ has
 trivial Lie algebra and hence is reduced.\footnote{An alternate way to show the group scheme $X$ used above is reduced
is to view it as a group functor, and observe that some power of
$\sigma$ is a power $F^m$ of the Frobenius morphism---see the very general
argument given by  in \cite[p.37]{Stein68}.  The comorphism $F^{*m}$ of $F^m$
is a power of the $p$th power map on the coordinate ring ${\mathbb F}_p[G]$. One can use this fact to show that, taking $m\gg 0$, $F^{*m}$ is simultaneously the identity on $k[X]$ and yet sends the radical of this finite dimensional algebra to zero. It follows the radical is zero, and $k[X]$ is reduced. 
}

 Consequently, $\ind_{\Delta\cap\Sigma}^\Sigma k\cong\ind_{G(\sigma)}^Gk$, if $G$ acts on the left hand side
 through the obvious map $G\to \Sigma$ and inverse of the above isomorphism $\Sigma\to G$. However,
 by (\ref{Mackey}), $\res_\Sigma^{G\times G}\ind^{G\times G}_\Delta\cong \ind^\Sigma_{\Sigma\cap \Delta}$, and the proposition follows.
  \end{proof}

In case $\sigma$ is a Frobenius morphism, the above result is stated without proof in \cite[1.4]{Hum06}.\footnote{We thank Jim Humphreys for some discussion on this point.}


\subsection{Passage from $G(\sigma)$ to $G$} Set $\sG_\sigma(k):=\ind_{G(\sigma)}^Gk$, where $\sigma:G\to G$ is a strict endomorphism. The filtration ${\mathcal F}_\bullet$ of
the rational $G$-module $\sG_\sigma(k)$ arises from the increasing $G\times G$-module filtration ${\mathcal F}'_\bullet$ of $k[G]$ with
sections $\nopH^0(\gamma)\otimes\nopH^0(\gamma^\star)$. Since these latter modules are all co-standard modules for $G\times G$, their order in ${\mathcal F}'_\bullet$
can be rearranged (cf. \cite[Thm.~4.2]{PSS12}).  Thus, for  $b\geq 0$, there is a (finite dimensional)
$G$-submodule $\sG_{\sigma,b}(k)$ of $\sG_\sigma(k)$ which has an increasing (and complete) $G$-stable filtration with sections precisely the $\nopH^0(\gamma)\otimes
\nopH^0(\gamma^\star )^{(\sigma)}$ satisfying $\langle\gamma,\alpha_0^\vee\rangle\leq b$, and with each such $\gamma$ appearing with
multiplicity 1. We have 
$$ \sG_{\sigma,b}(k)/\sG_{\sigma,b-1}(k)\cong\bigoplus_{\lambda\in X^+,\  \langle\lambda,\alpha^\vee_0\rangle=b}\nopH^0(\lambda)\otimes
\nopH^0(\lambda^\star)^{(\sigma)}.$$
Now we can state the following basic result.

\begin{theorem}\label{existsAnF} Let $m$ be a nonnegative integer and $\sigma:G\to G$ be a strict endomorphism. Let $b\geq 6m+6h-8$ (which is independent of $p$ and $\sigma$).
 Then, for any $\lambda,\mu\in X_\sigma$, 
 \begin{equation}\label{decomposition}
\Ext^m_{G(\sigma)}(L(\lambda),L(\mu))\cong\Ext^m_G(L(\lambda),L(\mu)\otimes \sG_{\sigma,b}).\end{equation}
 In addition,
\begin{equation}\label{vanishing}
\Ext^n_G(L(\lambda),L(\mu)\otimes\nopH^0(\nu)\otimes\nopH^0(\nu^\star )^{(\sigma)})=0
\end{equation} 
for all $n\leq m$, $\nu\in X^+$, satisfying $\langle \nu,\alpha_0^\vee\rangle>b$. 
\end{theorem}

\begin{proof}  In case (I), the case of the Chevalley groups, this result is proved in \cite[Thm.~4.4]{PSS12}.
Very little modification is needed in case (II), the case of the Steinberg groups, because in this
case the infinitesimal subgroups $G_\sigma$ identify with ordinary Frobenius kernels $G_r$. Finally,
for case (III), the Ree and Suzuki groups, all  results given in Section 2 (specifically, the
spectral sequences (2.1.1), (2.1.2), and (2.1.3), and Theorem 2.3.1) can be applied to obtain
the required result. We leave further details to the reader.
\end{proof}


\section{Bounding cohomology of finite groups of Lie type} In this section, we prove Theorem 
\ref{thm:finitebound}.

\subsection{A preliminary lemma}\label{ss:prelim} We begin by proving a lemma which will enable us to find  universal bounds (independent of the prime) for extensions of 
irreducible modules for the finite groups $G(\sigma)$. 

The following lemma does not require the ${\mathbb F}_p$-splitting hypothesis of the notation section, but we reduce to that case in the first paragraph of the proof.  Note also that the proof appeals to the forthcoming Corollary \ref{lastcor} applied to a Frobenius kernel.  That result, which is demonstrated within the proof of Theorem \ref{thm:Cartan-Frobenius}, is a direct consequence of  \cite[Lem. 7.2]{PS11} and it is independent of the other sections in this paper.

\begin{lemma}\label{compFacts}  For positive integers $e,b$ there exists a constant $f=f(e,b)=f(e,b,\Phi)$ with the following property. Suppose that $G$ is a simple, simply connected algebraic group over $k=\overline{{\mathbb F}_p}$ having root system $\Phi$. If
$\mu\in X_e$ and $\xi\in X^+$ satisfies $\langle \xi,\alpha_0^\vee\rangle<b$, then the (composition factor) length of the rational $G$-module 
$L(\mu)\otimes L(\xi)$ is at most $f(e,b)$.\end{lemma}

\begin{proof}
We will actually prove a stronger result. Namely, that there exists a constant $f(e,b)$ that bounds the length of $L(\mu) \otimes L(\zeta)$ as a module for the Frobenius kernel $G_{e}$. Clearly the $G$-length of a rational $G$-module is always less than or equal to its length after restriction to $G_e$.
Without loss of generality, we can always assume that $G$ is defined and split over ${\mathbb F}_p$.
This is a convenience which allows the use of familiar notation.

For any given prime $p$, it is clear that a bound exists (but depending on $p$) on the lengths since $|X_e|<\infty$, and there are a finite number of weights $\xi$ satisfying the condition $\langle\xi,\alpha_0^\vee\rangle<b$. 
Hence, it is sufficient to find a constant that uniformly bounds the number of composition factors of
all $L(\mu)\otimes L(\xi)$ for all sufficiently large $p$. 
By \cite{AJS94}, there is a positive integer $p_0\geq h$ such that
the Lusztig character formula holds for all $G$ with root system $\Phi$ provided the characteristic
$p$ of the defining field is at least $p_0$. In addition, it is assumed that $p \geq 2(h-1)$.

Let $Q_e(\mu)$ denote the $G_e$-injective hull of $L(\mu)$. Embed $L(\mu)\otimes L(\xi)$ in $Q_e(\mu)\otimes H^0(\xi)$ as a $G_e$-module and proceed to find a bound for the $G_e$-length of the latter module.  Corollary \ref{lastcor} applied to $G_e$ (or the proof of Theorem \ref{thm:Cartan-Frobenius}) provides a constant $k'(\Phi, e)$ that bounds the $G_e$-length of $Q_e(\mu)$ for all primes $p$ satisfying the above conditions. The dimension of any irreducible $G_e$-modules is at most the dimension of the $e$th Steinberg module $St_e$. It follows that $\dim (Q_e(\mu)\otimes H^0(\xi))/\dim St_e \leq k'(\Phi, e)\cdot \dim H^0(\xi)$.  Now $Q_e(\mu)\otimes H^0(\xi)$ decomposes into a direct sum of $Q_e(\omega)$, $\omega \in X_e$. The dimension of each $Q_e(\omega)$ that appears as a summand is a multiple of $\dim St_e$. Therefore, there are at most $k'(\Phi, e)\cdot \dim H^0(\xi)$ many summands, each having at most $k'(\Phi,e)$ many $G_e$-factors. Hence, the $G_e$-length of $Q_e(\mu)\otimes H^0(\xi)$ is bounded by $k'(\Phi, e)^2\cdot \dim H^0(\xi)$. Using Weyl's dimension formula, the numbers $\dim \nopH^0(\xi)$ (for $\xi$ satisfying $\langle\xi,\alpha_0^\vee\rangle<b$)
 are uniformly bounded by a constant $d=d(b)=d(b,\Phi)$.
\end{proof}

\begin{remark} In the presence of any strict endomorphism $\sigma:G\to G$, the set $X_e$ above
can be obviously replaced by the set $X_\sigma$ of $\sigma$-restricted weights, since $X_\sigma
\subseteq X_e$ for some $e$.\end{remark}


\subsection{Proof of Theorem \ref{thm:finitebound} }  By Theorem \ref{existsAnF},
$$E:=\Ext^m_{G(\sigma)}(L(\lambda),L(\mu))\cong\Ext^m_G(L(\lambda),L(\mu)\otimes \sG_{\sigma,b})$$
where $\sG_{\sigma,b}$ has composition factors $L(\zeta)\otimes L(\zeta')^{(\sigma)}$ with $\zeta, \zeta'$  in the set  
$\pi_{b-1}$, with $b:=6m +6h-8$.  Let $L(\xi)$ be a composition factor of $L(\mu)\otimes L(\zeta)$ for some $\zeta\in \pi_{b-1}$. Then, as $\mu\in X_e$ and $\zeta\in\pi_{b-1}$,  a direct
calculation (or using \cite[Lem. 2.1(b),(c)]{PSS12}), gives that $\xi\in\pi_{b'-1}$, where $b'=(p^e-1)(h-1)+b$. Choose a constant integer $e'=e'(e)$, independent of $p$ and $\sigma$, so that
$e'\geq[\log_p((p^e-1)(h-1) +b)]+1$. (If $p^e\geq b$, we can take $e'(e)
=e+[\log_2h] +1$.) 
Then $\xi$ is $p^{e'}$-restricted, by \cite[Lem. 2.1(a)]{PSS12}.

We need three more constants:
\begin{enumerate}\item By Theorem \ref{extalg}, there is a constant $c(\Phi,m,e')$ with the property that
 $$\dim \Ext_G^m(L(\tau),L(\xi))\leq c(\Phi,m,e'),\quad \forall \tau\in X^+,\,\, \forall\xi\in X_{e'}.$$
\item Set $s(\Phi,m)$ to be the maximum length of $\sG_{\sigma,b}$ over
all primes $p$---clearly, this number is finite; in fact, $\dim \sG_{\sigma,b}$ as a vector space is bounded, independently of $p,\sigma$, though its weights do depend on $p$ and $\sigma$.
\item By Lemma \ref{compFacts}, there is a constant $f=f(\Phi,e,b)$ bounding all the lengths of the
tensor products $L(\mu)\otimes L(\zeta)$ over all primes $p$, all $\mu\in X_e$ and all  $\zeta\in\pi_{b-1}$.\end{enumerate}

Now, since $\lambda\in X_\sigma^+$, we have $L(\lambda)\otimes L(\nu)^{(\sigma)}$ irreducible, thus

\begin{align*}
\dim E& = \dim\Ext^m_G(L(\lambda),L(\mu)\otimes \sG_{\sigma,b})\\
& \leq s(\Phi,m)\max_{\zeta,\nu\in \pi_{b-1}}\{ \dim\Ext^m_G(L(\lambda),L(\mu)\otimes L(\zeta)\otimes L(\nu^\star )^{(\sigma)})\}\\
& \leq s(\Phi,m) f(\Phi,m,e) \max_{\nu\in \pi_{b-1},\xi\in X_{e'}}\{\dim\Ext^m_G(L(\lambda)\otimes L(\nu)^{(\sigma)},L(\xi))\} \\
& \leq s(\Phi,m) f(\Phi,m,e) c(\Phi,m, e').
 \end{align*}
Since $e'$ is a function of $e$, $m$ and $\Phi$, we can take $D(\Phi,m,e)=s(\Phi,m)f(\Phi,m,e)c(\Phi,m,e')$, proving
the first assertion of the theorem.  For the final conclusion, take $\mu=0$ and replace $\lambda$ by $\lambda^\star$. 

This concludes the proof of Theorem
 \ref{thm:finitebound}. \qed

\medskip
\begin{remark} (a) The easier bounding of the integers $\dim\opH^m(G(\sigma),L(\lambda))$ over all $p,r,\lambda$ does
not require Lemma \ref{compFacts}. However, it does still require the established Lusztig
character formula for large $p$ (even though the final result holds for all $p$)
since the proofs in \cite{PS11}
{\it do} require the validity of the Lusztig character formula for $p$ large.

 (b) Following work of Parshall and Scott \cite{PS12}, but using finite groups of Lie type in place of their algebraic group counterparts one can investigate the following question. 
For a given root system $\Phi$ and non-negative integer $m$, let $D(\Phi,m)$ be the least upper bound of the integers $\dim \opH^m(G(\sigma),L(\lambda))$ over $\sigma$ and
all $\sigma$-restricted dominant weights $\lambda$. Then one can ask for the rate of growth of the
sequence $\{D(\Phi,m)\}$.   In the rank 1 case (i.e., $SL_2$), it is known  from results of Stewart \cite{Ste12} that the growth
rate can be exponential even in the rational cohomology case. However, the corresponding question remains open for higher ranks.

(c) One could ask if the condition on $e$ in the theorem is necessary to bound the dimension of the $\Ext^m$-groups for $i\geq 2$. Bendel, Nakano and Pillen \cite[Thm.~5.6]{BNP06} show that one can 
drop the condition in case $m=1$ (see also \cite[Cor. 5.3]{PS11}). However, in \cite[Thm.~1]{Ste12} a sequence of irreducible modules $\{L_r\}$ was given for any simple group $G$ for $p$ sufficiently 
large showing that $\dim\Ext^2_G(L_r,L_r)\geq r-1$. One can see the same examples work at least for all finite Chevalley groups. This demonstrates that the condition on $e$ is necessary in the above theorem also. 
\end{remark}


\section{Bounding cohomology of Frobenius kernels}\label{S:FrobeniusKernels} 

This section proves Theorem \ref{thm:boundforFrobeniusKernel}, an analogue
of Theorem \ref{thm:finitebound} for Frobenius kernels.  The result is stated in the 
general context of $G_{\sigma}$ for a surjective endomorphism $\sigma: G \to G$.  Recall from the 
discussion in \S2.3 that $G_{\sigma}$ is either an ordinary Frobenius kernel $G_r$ (for a non-negative integer $r$), or $G_{r/2}$ for an odd positive integer $r$, in the cases of the Ree and Suzuki groups. 

\subsection{Induction from infinitesimal subgroups} Analogous to the previous use of the induction functor 
$\ind_{G(\sigma)}^{G} -$, we consider the induction functor $\text{ind}_{G_{\sigma}}^{G} -$. 
This functor is exact since $G/G_{\sigma}$ is affine. When this functor is applied to the trivial module, there are the following identifications of $G$-modules: 
$$\text{ind}_{G_{\sigma}}^{G}k\cong k[G/G_{\sigma}]\cong k[G]^{(\sigma)},$$ 
where the action of $G$ on the right hand side is via the left regular representation
(twisted).  
As noted in Lemma \ref{filtration}, $k[G]$ as a $G\times G$-bimodule (with the left and right
regular representations respectively) has a filtration with sections of the form $\nopH^{0}(\nu)\otimes \nopH^{0}(\nu^\star)$, $\nu\in X^+$ with 
each $\nu$ occurring precisely once. Note that this is an exterior tensor product
with each copy of $G$ acting naturally on the respective induced modules and trivially
on the other.  Hence, $k[G]^{(\sigma)}$ has a $G\times G$-filtration with sections $\nopH^{0}(\nu)^{(\sigma)}\otimes \nopH^{0}(\nu^\star )^{(\sigma)}$.  
By restricting the action of $G \times G$ on  $k[G]^{(\sigma)}$ to the first (left hand) $G$-factor,
we conclude that $k[G]^{(\sigma)}$ with the (twisted) left regular action, and hence $\ind_{G_{\sigma}}^Gk$, admits a filtration with sections of the form 
$(\nopH^0(\nu)^{(\sigma)})^{\oplus \dim\nopH^0(\nu)}$.

We can now apply generalized Frobenius reciprocity and this fact to obtain the following inequality: 
\begin{eqnarray}\label{E:firstdimbound} 
\dim\text{Ext}^{m}_{G_{\sigma}}(L(\lambda),L(\mu))&=& \dim \text{Ext}^{m}_{G}(L(\lambda),L(\mu)\otimes \text{ind}_{G_{\sigma}}^{G} k) \notag\\
&\leq& \sum_{\nu\in X^+}\dim\Ext^m_G(L(\la),L(\mu)\otimes(\nopH^0(\nu)^{(\sigma)})^{\oplus \dim\nopH^0(\nu)})\\
&\leq& \sum_{\nu\in X^+} \dim \text{Ext}_{G}^{m}(L(\lambda),L(\mu)\otimes \nopH^{0}(\nu)^{(\sigma)})\cdot \dim \nopH^{0}(\nu) \nonumber.
\end{eqnarray}


\subsection{Proof of Theorem \ref{thm:boundforFrobeniusKernel}}
Letting $s(m)$ be as in Theorem~\ref{thm:Gr-weightbound}, form the finite set
$$X(\Phi,m):=\{\tau\in X^+:\ \langle \tau,\alpha_{0}^{\vee} \rangle < s(m)\}$$
of dominant weights which depends only on $\Phi$ and $m$.  Necessarily,
$X(\Phi,m)$ is a saturated subset (i.~e.,an ideal) of $X^+$.
 
 Let $\lambda\in X_e\subseteq X_\sigma$ and $\mu\in X_\sigma$. If $\text{Ext}^{m}_{G}(L(\lambda),L(\mu)\otimes \nopH^{0}(\nu)^{(\sigma)})\neq 0$, 
the Hochschild-Serre spectral sequence
$$E_{2}^{i,j}=\text{Ext}^{i}_{G/G_{\sigma}}(V(\nu^{\star})^{(\sigma)},\text{Ext}^{j}_{G_{\sigma}}(L(\lambda),L(\mu)))\Rightarrow \text{Ext}^{i+j=m}_{G}(L(\lambda),L(\mu)\otimes \nopH^{0}(\nu)^{(\sigma)})$$
implies there exists $i,j$ such that $i+j=m$ and 
$$\text{Ext}^{i}_{G/G_{\sigma}}(V(\nu^{\star})^{(\sigma)},\text{Ext}_{G_{\sigma}}^{j}(L(\lambda),L(\mu)))\neq 0.$$ 
By Theorem~\ref{thm:Gr-weightbound}, if $[\Ext^j_{G_\sigma}(L(\lambda),L(\mu))^{(-\sigma)}:L(\gamma)]\not=0$,
 then
$\gamma\in\pi_{s(j)}$, and so  $\langle \gamma,\alpha_{0}^{\vee} \rangle < s(j) \leq s(m).$ However,
 if $\text{Ext}^{i}_{G/G_{\sigma}}(V(\nu^{\star})^{(\sigma)},L(\gamma)^{(\sigma)})\cong 
\text{Ext}^{i}_{G}(V(\nu^{\star}),L(\gamma))\neq 0$ then 
$\nu^\star\leq \gamma$. Thus, $\nu^{\star}\in X(\Phi,m)$ and so $\nu\in X(\Phi,m)$. 

The inequality (\ref{E:firstdimbound}) and Theorem~\ref{extalg} now give 
\begin{align*} 
\dim\Ext^{m}_{G_{\sigma}}&(L(\lambda),L(\mu)) \\
&\leq \sum_{\nu\in X(\Phi,m)} \dim \Ext_{G}^{m}(L(\lambda),L(\mu)\otimes \opH^{0}(\nu)^{(\sigma)})\cdot \dim \nopH^{0}(\nu) \qquad (\text{by }(\ref{E:firstdimbound}))\\
&\leq \sum_{\nu\in X(\Phi,m)} \sum_{\tau\in X^+} \dim \text{Ext}_{G}^{m}(L(\lambda),L(\mu)\otimes L(\tau)^{(\sigma)})\cdot [\nopH^{0}(\nu):L(\tau)]\cdot \dim \nopH^{0}(\nu) \\
&\leq c(\Phi,m,e) \sum_{\nu\in X(\Phi,m)} \sum_{\tau\in X^+}  [\nopH^{0}(\nu):L(\tau)]\cdot \dim \nopH^{0}(\nu) \qquad (\text{by Theorem \ref{extalg}})\\
&\leq c(\Phi,m,e)\sum_{\nu\in X(\Phi,m)} (\dim \nopH^{0}(\nu))^{2}.
\end{align*} 

Since $|X(\Phi,m)|<\infty$ and the numbers $\dim\nopH^0(\nu)$ are given by
Weyl's dimension formula, the first claim of the theorem is proved, putting
\[ E(\Phi,m,e):=c(\Phi,m,e)\sum_{\nu\in X(\Phi,m)} (\dim \nopH^{0}(\nu))^{2}.\]
 For the second claim, set $\mu = 0$ and replace $\la$ with $\la^\star $. Then, in the above argument, apply Theorem 
\ref{extalg} and replace $c(\Phi,m,e)$ with $c(\Phi,m,0)$. 
Similarly, for the last claim, apply Theorem \ref{extalgm=1} to replace
$c(\Phi,1,e)$ by $c(\Phi)$.  \qed


\subsection{Examples} \label{examples:boundforFrobenius}We will illustrate Theorem \ref{thm:boundforFrobeniusKernel} with some examples for ordinary Frobenius kernels $G_r$.  First,  the theorem says that the dimension of $G_r$-cohomology groups (in some fixed degree) of irreducible modules 
can be bounded independently of $r$.  In low degrees, one can explicitly see that the dimension
of the cohomology of the trivial module is independent of $r$.  On the other hand, in
degree 2, one sees that the dimension is clearly dependent on the root system.

\begin{example}\label{Example-trivialmodule} Assume that the Lie algebra ${\mathfrak g}$ of 
$G$ is simple (or assume that $p \neq 2, 3$ for certain root systems). Then 
$\operatorname{H}^{1}(G_{r},k)=0$ for all $r\geq 1$ (cf. \cite{HHA84}). Furthermore, 
$\operatorname{H}^{2}(G_{r},k)\cong \operatorname{Ext}^{2}_{G_{r}}(k,k)\cong 
({\mathfrak g}^*)^{(r)}$ for all $r \geq 1$ (cf. \cite{BNP07}).  
\end{example} 

On the other hand, the following example demonstrates that the dimension of $\Ext$-groups
between arbitrary irreducible modules (as in the first part) of the theorem cannot be bounded by
a constant independent of $r$.  In particular, one can have $\Ext$-groups of arbitrarily 
high dimension.

\begin{example}\label{ExampleforH2} Let $G=SL_2$ with let $p>2$. Set $\lambda=1+p+p^{2}+\cdots +p^{r}$, and let $L(\lambda)=L(1)\otimes L(1)^{(1)}\otimes L(1)^{(2)}\otimes \dots \otimes L(1)^{(r)}$. 
From \cite[Thm.~1]{Ste12}, we have $\dim\Ext^2_G(L(\lambda),L(\lambda))=r$. Assume that $s\geq r$.  Applying the Hochschild-Serre spectral sequence to $G_s \unlhd G$ and fact that the $E_{2}$-term is a 
subquotient of the cohomology, 
we see that 
\begin{align*}
r=\dim\Ext^2_G(L(\lambda),L(\lambda))\leq &\dim\Ext^2_G(k,\Hom_{G_s}(L(\lambda),L(\lambda))^{(-s)})\\&
+\dim\Ext^1_G(k,\Ext^1_{G_s}(L(\lambda),L(\lambda))^{(-s)})\\
&+\dim\Hom_G(k,\Ext^2_{G_s}(L(\lambda),L(\lambda))^{(-s)}).
\end{align*}
But $\Hom_{G_s}(L(\lambda),L(\lambda))^{[-s]}\cong k$ and so the first term on the right hand side is 0 by \cite[II.4.14]{Jantzen}. 
Also, \cite[Thm.~4.5]{HHA84} yields that  $\Ext^1_{G_s}(L(\lambda),L(\lambda))^{(-s)}=0$, so that the second term on the right hand side is also 0. 
Thus \[r\leq\dim \Hom_G(k,\Ext^2_{G_s}(L(\lambda),L(\lambda))^{(-s)})\leq \dim \Ext^2_{G_s}(L(\lambda),L(\lambda)).\] 
\end{example} 

Returning to the case of cohomology, the following example suggests
how the dimension of $\opH^1(G_r,L(\la))$ may depend on the root system.  
Theorem \ref{thm:lowerboundforH1} in the following subsection will expand on this.  

\begin{example} \label{A_n-example} Let $G=SL_{n+1}$ with $\Phi$ of type $A_{n}$. 
We will assume that $p>n+1$ so that $0$ is a 
regular weight. Consider the dominant weights of the form 
$\lambda_{j}=p^r\omega_{j}-p^{r-1}\alpha_{j}$ where $j=1,2,\dots,n$. According to 
\cite[Thm.~3.1]{BNP04-Frob} these are the minimal dominant weights $\nu$ such that $\text{H}^{1}(G_{r},\nopH^{0}(\nu))\neq 0$. Furthermore, 
$\text{H}^{1}(G_{r},\nopH^{0}(\lambda_{j}))\cong L(\omega_{j})^{(r)}$ for each $j$. 

Since $p>n+1$ the weights $\lambda_{j}$ are not in the root lattice and cannot be linked under the action of the affine Weyl group to $0$, thus 
any $W_{p}$-conjugate to $\lambda_{j}$ (under the dot action) cannot be linked to $0$. It follows that if $\mu\in X^+$ and $\mu \uparrow \lambda_{j}$ then 
$\text{H}^{1}(G_{r},L(\mu))=0$. This can be seen by using induction on the ordering of the weights and the long exact sequence 
induced from the short exact sequence $0\rightarrow L(\mu)\rightarrow \nopH^{0}(\mu)\rightarrow N\rightarrow 0$. Note that $N$ has composition 
factors which are strongly linked and less than $\mu$. Moreover $N$ has no trivial $G_r$-composition factors by using linkage and the fact that 
$\mu< \lambda_{j}$. 

Now consider the short exact sequence $0\rightarrow L(\lambda_{j})\rightarrow \nopH^{0}(\lambda_{j})\rightarrow M\rightarrow 0$. The long exact sequence 
and the fact that $\text{H}^{1}(G_{r},M)=0$ yields a short exact sequence of the form: 
$$0\rightarrow \text{H}^{0}(G_{r},M)\rightarrow \text{H}^{1}(G_{r},L(\lambda_{j}))\rightarrow \text{H}^{1}(G_{r},\nopH^{0}(\lambda_{j}))\rightarrow 0.$$
But as before, we have $\text{H}^{0}(G_{r},M)=0$, thus 
\begin{equation}
\operatorname{H}^{1}(G_{r},L(\lambda_{j})) \cong \opH^1(G_r,\nopH^0(\lambda_j)) \cong L(\omega_{j})^{(r)}
\end{equation}
for all $j=1,2,\dots,n$. 
\end{example}


\subsection{A lower bound on the dimension of first cohomology}\label{lowerBound} In this section we extend Example~\ref{A_n-example} by showing that the dimension of the cohomology group $\text{H}^{1}(G_{r},L(\lambda))$ cannot be universally bounded independent of the root system. This result indicates that the Guralnick Conjecture \cite{Gur86} on a universal bound
for the first cohomology of finite groups cannot hold for arbitrary finite group schemes.  

\begin{theorem}\label{thm:lowerboundforH1}  Let $G$ be a simple, simply connected algebraic group and $r$ be a non-negative integer.  The inequality 
$$\operatorname{max}\{\dim \operatorname{H}^{1}(G_{r},L(\lambda)):\ \lambda\in X_r\}\geq \dim V$$ holds,
where $V$ is the irreducible non-trivial finite dimensional $G$-module of smallest dimension. 
\end{theorem} 

\begin{proof} Suppose that all $G/G_{r}$-composition factors of 
$\text{H}^{1}(G_r,L(\lambda))$ are trivial for all $\lambda\in X_r$. Then one could conclude that,
for any finite dimensional $G$-module $M$, the $G/G_{r}$-structure on 
$\text{H}^{1}(G_{r},M)$ is either 
a direct sum of trivial modules or $0$. This can be seen by using induction on the composition length of $M$, the long exact sequence in cohomology associated to a short
exact sequence of modules, and the fact that $\text{Ext}^{1}_{G}(k,k)=0$. 

However, by \cite[Thm.~3.1(A-C)]{BNP04-Frob}, there exist a finite dimensional 
$G$-module (of the form $H^0(\mu)$) whose $G_{r}$-cohomology has non-trivial 
$G/G_{r}$-composition factors, thus giving a contradiction.  
Therefore, there must exist a $\lambda\in X_r$ (necessarily non-zero) such 
that $\text{H}^{1}(G_{r},L(\lambda))$ has a non-trivial $G/G_{r}$-composition factor. 
\end{proof} 

Example~\ref{A_n-example} illustrates that one can realize $\dim \operatorname{H}^{1}(G_{r},L(\lambda))$  
as the dimension of a (non-trivial) minimal dimensional irreducible representation in type $A_{n}$ for some $\lambda \in X_{r}$ when $p>n+1$. An interesting question would be to explicitly realize the smallest dimensional  
non-trivial representation in general as $\operatorname{H}^{1}(G_{r},L(\lambda))$ for some $\lambda$.


\section{Cartan Invariants}\label{cartan}

In either the finite group or the infinitesimal group setting, the determination of Cartan
invariants---the multiplicities $[P:L]$ of irreducible modules $L$ in projective indecomposable
modules $P$ (PIMs)---is a classic 
representation theory problem.  In this section we observe that, for $G(\sigma)$ or 
$G_\sigma$, these numbers (in the defining characteristic case) can be bounded by a constant depending on the root system $\Phi$ 
and the height $r$ of $\sigma$, independently of the characteristic. In the process we will see that there is also a bound for the composition series length of $P$. Since the Ree and Suzuki groups only involve the primes 2 and 3,
 those cases can be ignored. Thus, we can assume that $G_\sigma=G_r$ for a
positive integer $r$.

\subsection{Cartan invariants for Frobenius kernels}
For $\la \in X_r$, let $Q_r(\la)$ denote the $G_r$-injective hull of $L(\la)$. 
In the category of finite dimensional $G_r$-modules, injective modules are
projective (and vice versa), and the projective indecomposable modules (PIMs) consist
precisely of the $\{Q_r(\la) : \la \in X_r\}$.

\begin{theorem}\label{thm:Cartan-Frobenius} Given a finite irreducible root system $\Phi$ and
a positive integer $r$, there is a constant $K(\Phi,r)$ with the following property. Let $G$ be a simple, simply connected  algebraic group over an algebraically closed field $k$ of positive characteristic with 
irreducible root system $\Phi$, and let $\sigma$ be a strict endomorphism of $G$ of height $r$. Then
$$
[Q_r(\la):L(\mu)|_{G_r}] \leq K(\Phi,r),
$$
for all $\la, \mu \in X_r$.
\end{theorem}

\begin{proof}
Assume that $p \geq 2(h-1)$. Then, for any $\la \in X_r$, 
$Q_r(\la)$ admits a unique rational $G$-module structure which restricts to the original $G_r$-structure
\cite[\S II 11.11]{Jantzen}. 
By \cite[Lem. 7.2]{PS11}, the number of $G$-composition factors of $Q_r(\la)$
is bounded by some constant $k(\Phi,r)$.  
The irreducible $G$-composition factors of $Q_r(\la)$ are of the form 
$L(\mu_0 + p^r\mu_1)$ with $\mu_0 \in X_r$ and $\mu_1 <_Q 2\rho$. 
As a $G_r$-module, $L(\mu_0 + p^r\mu_1) \cong L(\mu_0)\otimes L(\mu_1)^{(r)} \cong
L(\mu_0)^{\oplus\dim L(\mu_1)}$.  
Since $\mu_1 <_Q 2\rho$, the dimensions of all possible $L(\mu_1)$ are  
bounded by some number $d(\Phi)$, depending only on $\Phi$.   
Therefore, the $G_r$-composition length of $Q_r(\la)$ (for any $\la \in X_r$)
is bounded by $k(\Phi, r) \cdot d(\Phi)$. 
This number necessarily bounds all $[Q_r(\la) : L(\mu)|_{G_r}]$. 

This leaves us finitely many primes $p < 2(h-1)$. In general, we have
$[Q_r(\la) : L(\mu)|_{G_r}] \leq \dim Q_r(\la)$.  For a given root system
$\Phi$ and positive integer $r$, the $G_r$-composition length of $Q_r(\la)$ is bounded by, for instance, 
$\operatorname{max}\{\dim Q_r(\nu) : \nu \in X_r,\ p < 2(h-1)\}$. Combining these
cases gives the claimed bound $K(\Phi,r)$.
\end{proof}

\begin{remark} As noted in Section \ref{ss:prelim}, the preceding proof does {\it not} make use of any of the preceding results of this paper.   The reader may also recall that this proof is in fact {\it required} in the proof of Lemma \ref{compFacts}.  On the other hand, the proof of Theorem \ref{thm:Cartan-finite}, which is given in
the next section, {\it does} require Lemma \ref{compFacts}.
\end{remark}


\subsection{Cartan invariants of finite groups of Lie type;  proof of Theorem \ref{thm:Cartan-finite}}As noted above, we can assume that
$\sigma=F^r$ or $\sigma=F^r\circ\theta=\theta\circ F^r$. For the
 finite group $G(\sigma)$, the PIMs are
again in one-to-one correspondence with the irreducible modules, i.~e., simply with
the set $X_\sigma=X_r$. Let $U_\sigma(\la)$ denote the projective
cover of $L(\la)$ for $\la \in X_r$ in the category of $kG(\sigma)$-modules.  
As noted in the proof of Theorem \ref{thm:Cartan-Frobenius},
when $p \geq 2(h-1)$, each PIM $Q_r(\la)$ in the category of $G_r$-modules admits a unique $G$-structure.  Upon restriction 
to $G(\sigma)$, $Q_r(\la)$ remains injective (or equivalently projective).\footnote{This follows here by
simply observing that $Q_r(\lambda)$ is a direct summand ${\text{\rm St}}_r\otimes L(\lambda')$, where
${\text{\rm St}}_r$ is the $r$th Steinberg module and $\lambda'\in X^+$ \cite[II.11.1]{Jantzen}.}  
Hence $U_\sigma(\la)$ is a direct summand of $Q_r(\la)$.  As shown below, this allows us to 
modify the argument for Frobenius kernels to obtain an analogous result for
the $G(\sigma)$. Note that in  \cite{Pil95}, Pillen showed (in the case (I) of Chevalley groups) that the
``first'' Cartan invariant $[U_r(0) : k|_{\gfpr}]$ is independent of $p$ for large $p$.

Assume that $p \geq 2(h-1)$.  As in the proof of Theorem \ref{thm:Cartan-Frobenius},
if one can bound $[U_\sigma(\la) : L(\mu)|_{G(\sigma)}]$ in this setting, then one can deal with the 
finitely many remaining primes.  Since $U_\sigma(\la)$ is a summand of $Q_r(\la)|_{G(\sigma)}$, it suffices 
to bound $[Q_r(\la)|_{G(\sigma)} : L(\mu)|_{G(\sigma)}]$, that is, the composition multiplicity of the restriction of $L(\mu)$ 
to $G(\sigma)$ as a $G(\sigma)$-composition factor of $Q_r(\la)|_{G(\sigma)}$. To do this, we follow the argument 
in the proof of Theorem \ref{thm:Cartan-Frobenius}.

As above, the number of $G$-composition factors
of $Q_r(\la)$ is bounded by $k(\Phi,r)$ and the $G$-composition
factors have the form $L(\mu_0 + \sigma^*\mu_1) \cong L(\mu_0)\otimes L(\mu_1)^{(\sigma)}$
for $\mu_0 \in X_r$ and $\mu_1 <_Q 2\rho$.  As a $G(\sigma)$-module (as opposed
to a $G_r=G_\sigma$-module), $L(\mu_0 + \sigma^*\mu_1) \cong L(\mu_0) \otimes  L(\mu_1)$. 
By Lemma \ref{compFacts}, since $\mu_1 <_Q 2\rho$, the number of $G$-composition 
factors of $L(\mu_0) \otimes  L(\mu_1)$ is bounded by some number $f(\Phi,r)$, 
independent of $p$. (Take $f(\Phi,r)=f(e,b)$  with $e=r$ and $b=2(h-1)$ in Lemma \ref{compFacts}.) Therefore, one $G$-factor of the form $L(\mu_0 + \sigma^*\mu_1)$ could 
give rise to at most $f(\Phi,r)$ many $G(\sigma)$-sections $L(\nu)|_{G(\sigma)}$, $\nu\leq \mu_0+\mu_1
=\mu_0+\sigma^*\mu_1-(\sigma^*-1)\mu_1$. 
However, it could happen that some of the $\nu $ are not $p^r$-restricted, and we might have to iterate this process, first replacing $\nu=\nu_0+\sigma^*\nu_1$ by $\nu_0+\nu_1=\nu_0+\sigma^*\nu_1-(\sigma^*-1)\nu_1$.
 The reader may check that after at most $2(h-1)$ iterations, we get only weights that are $p^r$-restricted.
   Consequently, the 
$G(\sigma)$-length of $Q_r(\la)$ is bounded by $k(\Phi, r)\cdot f(\Phi,r)^{2
(h-1)}$, thus giving a bound on all $[Q_r(\la) : L(\mu)|_{G(\sigma)}]$, as desired. This concludes the proof of Theorem \ref{thm:Cartan-finite}.  \qed

\medskip
As a consequence of the above proofs, we have the following result.  

\begin{cor}\label{lastcor} There exists a constant $k'(\Phi,r)$ depending only on the irreducible root system and
the positive integer $r$ with the following property. If $P$ is a PIM
for $G_r$ or $G(\sigma)$ (in the defining characteristic) for a simple, simply connected algebraic
group $G$ over an algebraically closed field $k$ with root system $\Phi$, then the composition factor
length of $P$ is bounded by $k'(\Phi,r)$. Here $\sigma$ is any strict endomorphism of height $r$. 
\end{cor}

Obviously, $k'(\Phi,r)$ can be used as the constant in \cite[Lem. 7.2]{PS11} bounding the $G$-composition
length there, though this  latter result  and constant (denoted $k(\Phi,r)$ above) are used in the proof of the corollary.

Finally, both the corollary and Theorem \ref{thm:Cartan-finite} show that, once the root system $\Phi$ and $r$ are fixed, the Cartan invariants of
the finite groups $G(\sigma)$ are bounded, independently of the prime $p$. This answers the 
question of Hiss stated in Question \ref{hissq} (strong version), in the special case when $H=G(\sigma)$.

\let\section=\oldsection

\bibliographystyle{amsalpha}
\bibliography{BNPPSSbib}   

\providecommand{\bysame}{\leavevmode\hbox to3em{\hrulefill}\thinspace}
\providecommand{\MR}{\relax\ifhmode\unskip\space\fi MR }
\providecommand{\MRhref}[2]{%
  \href{http://www.ams.org/mathscinet-getitem?mr=#1}{#2}
}
\providecommand{\href}[2]{#2}
\begin{thebibliography}{GKKL08}

\bibitem[AJS94]{AJS94}
Henning Andersen, Jens Jantzen, and Wolfgang Soergel, \emph{Representations of
  quantum groups at a $p$th root of unity and of semisimple groups in
  characteristic $p$}, vol. 220, Ast\'erique, 1994.

\bibitem[And84]{HHA84}
Henning Andersen, \emph{Extensions of modules for algebraic groups}, Amer. J.
  Math. \textbf{106} (1984), 489--504.

\bibitem[BNP04]{BNP04-Frob}
Christopher Bendel, Daniel Nakano, and Cornelius Pillen, \emph{Extensions for
  {F}robenius kernels}, J. Algebra \textbf{272} (2004), no.~2, 476--511.

\bibitem[BNP06]{BNP06}
\bysame, \emph{Extensions for finite groups of {L}ie type. {II}. {F}iltering
  the truncated induction functor}, Representations of algebraic groups,
  quantum groups, and {L}ie algebras, Contemp. Math., vol. 413, American
  Mathematical Society, Providence, RI, 2006, pp.~1--23.

\bibitem[BNP07]{BNP07}
\bysame, \emph{Second cohomology groups for {F}robenius kernels and related
  structures}, Adv. Math. \textbf{209} (2007), no.~1, 162--197.

\bibitem[BNP11]{BNP11}
\bysame, \emph{On the vanishing ranges for the cohomology of finite groups of
  {L}ie type}, Int. Math. Res. Not. \textbf{doi:10.1093/imrn/rnr130} (2011).

\bibitem[BR03]{BR03}
C\'edric Bonaf\'e and Rapha\"el Rouquier, \emph{Cat\'egories d\'eriv\'ees et
  vari\'et\'es de {D}eligne-{L}usztig}, Publ. Math. I.H.E.S. \textbf{97}
  (2003), 1--59.

\bibitem[CPS83]{CPS83}
Edward Cline, Brian Parshall, and Leonard Scott, \emph{A {M}ackey imprimitivity
  theory for algebraic groups}, Math. Z. \textbf{182} (1983), no.~4, 447--471.

\bibitem[CPS09]{CPS09}
\bysame, \emph{Reduced standard modules and cohomology}, Trans. Amer. Math.
  Soc. \textbf{361} (2009), no.~10, 5223--5261.

\bibitem[GH98]{GH98}
Robert Guralnick and Corneliu Hoffman, \emph{The first cohomology group and
  generation of simple groups}, Groups and geometries ({S}iena, 1996), Trends
  Math., Birkh{\"a}user, Basel, 1998, pp.~81--89.

\bibitem[GKKL07]{GKKL07}
Robert Guralnick, William Kantor, Martin Kassabov, and Alexander Lubotzky,
  \emph{Presentations of finite simple groups: profinite and cohomological
  approaches}, Groups Geom. Dyn. \textbf{1} (2007), no.~4, 469--523.

\bibitem[GKKL08]{GKKL08}
\bysame, \emph{Presentations of finite simple groups: a quantitative approach},
  J. Amer. Math. Soc. \textbf{21} (2008), no.~3, 711--774.

\bibitem[GLS98]{GLS98}
Daniel Gorenstein, Richard Lyons, and Ronald Solomon, \emph{The classification
  of the finite simple groups, number 3}, vol.~40, Mathematical Surveys and
  Monographs, no.~3, American Mathematical Society, 1998.

\bibitem[GT11]{GT11}
Robert Guralnick and Pham Tiep, \emph{First cohomology groups of {C}hevalley
  groups in cross characteristic}, Ann. of Math. (2) \textbf{174} (2011),
  no.~1, 543--559.

\bibitem[Gur86]{Gur86}
Robert Guralnick, \emph{The dimension of the first cohomology groups}, Lecture
  Notes in Mathematics \textbf{1178} (1986), 94--97.

\bibitem[His00]{His00}
Gerhard Hiss, \emph{On a question of {B}rauer in modular representation theory
  of finite groups}, S\=urikaisekikenky\=usho K\=oky\=uroku (2000), no.~1149,
  21--29, Representation theory of finite groups and related topics (Japanese)
  (Kyoto, 1998).

\bibitem[Hum06]{Hum06}
James Humphreys, \emph{Modular representations of finite groups of {L}ie type},
  London Mathematical Society Lecture Note Series, vol. 326, Cambridge
  University Press, Cambridge, 2006.

\bibitem[Jan03]{Jantzen}
Jens Jantzen, \emph{Representations of algebraic groups}, second ed.,
  Mathematical Surveys and Monographs, vol. 107, American Mathematical Society,
  Providence, RI, 2003.

\bibitem[Kop84]{Kop84}
Markku Koppinen, \emph{Good bimodule filtrations for coordinate rings}, J.
  London Math. Soc. \textbf{30} (1984), no.~2, 244--250.

\bibitem[Pil95]{Pil95}
Cornelius Pillen, \emph{The first {C}artan invariant of a finite group of {L}ie
  type for large $p$}, J. Algebra \textbf{174} (1995), 934--947.

\bibitem[PS11]{PS11}
Brian Parshall and Leonard Scott, \emph{Bounding {E}xt for modules for
  algebraic groups, finite groups and quantum groups}, Adv. Math. \textbf{226}
  (2011), no.~3, 2065--2088.

\bibitem[PS12]{PS12}
\bysame, \emph{Cohomological growth rates and {K}azhdan-{L}usztig polynomials},
  Israel J. Math. (2012), in press.

\bibitem[PSS12]{PSS12}
Brian Parshall, Leonard Scott, and David Stewart, \emph{Shifted generic
  cohomology}, ArXiv:1205.1207, 2012.

\bibitem[Sco03]{Sco03}
Leonard Scott, \emph{Some new examples in 1-cohomology}, J. Algebra
  \textbf{260} (2003), no.~1, 416--425.

\bibitem[Sin94]{Sin94}
Peter Sin, \emph{Extensions of simple modules for special algebraic groups}, J.
  Algebra \textbf{170} (1994), no.~3, 1011--1034.

\bibitem[SS70]{SpSt70}
Tony Springer and Robert Steinberg, \emph{Conjugacy classes}, Lecture Notes in
  Mathematics \textbf{131} (1970), 167--266.

\bibitem[Ste68]{Stein68}
Robert Steinberg, \emph{Endomorphisms of linear algebraic groups}, Mem. Amer.
  Math. Soc., vol.~80, American Mathematical Society, 1968.

\bibitem[Ste12]{Ste12}
David Stewart, \emph{Unbounding {E}xt}, J. Algebra \textbf{365} (2012), 1--11.

\end{thebibliography}

  \end{document}